\documentclass[11pt]{amsart}
\usepackage{amssymb}
\usepackage{amsmath}
\usepackage{color}
\usepackage[all]{xy}
\newtheorem{theorem}{Theorem}[section]
\newtheorem{lemma}[theorem]{Lemma}
\newtheorem{remark}[theorem]{Remark}
\newtheorem{definition}[theorem]{Definition}
\newtheorem{proposition}[theorem]{Proposition}

\def\C{\mathbb{C}}
\def\zC{\mathbb{C}}
\def\N{\mathbb{N}}
\def\zN{\mathbb{N}}
\def\T{\mathbb{T}}

\begin{document}

\title[Hypercyclic behavior of some non-convolution operators]{Hypercyclic behavior of some non-convolution\\ operators on $H(\C^N)$}

\author{Santiago Muro, Dami\'an Pinasco, Mart\'in Savransky}

\thanks{Partially supported by PIP 2010-2012 GI 11220090100624, PICT 2011-1456, UBACyT 20020100100746, ANPCyT PICT 11-0738 and CONICET}

\address{Santiago Muro
\hfill\break\indent Departamento de Matem\'{a}tica - Pab I,
Facultad de Cs. Exactas y Naturales, Universidad de Buenos Aires,
(1428), Ciudad Aut\'onoma de Buenos Aires, Argentina and CONICET} \email{{\tt smuro@dm.uba.ar}}

\address{Dami\'an Pinasco
\hfill\break\indent Departamento de Matem\'aticas y Estad\'{\i}stica,
Universidad Torcuato Di Tella, Av. Figueroa Alcorta 7350, (1428), Ciudad Aut\'onoma de Buenos Aires, Argentina and CONICET}
\email{{\tt dpinasco@utdt.edu}}

\address{Mart\'{i}n Savransky
\hfill\break\indent Departamento de Matem\'{a}tica - Pab I,
Facultad de Cs. Exactas y Naturales, Universidad de Buenos Aires,
(1428), Ciudad Aut\'onoma de Buenos Aires, Argentina and CONICET} \email{{\tt msavran@dm.uba.ar}}


\begin{abstract}
We study hypercyclicity properties of a family of non-convolution operators defined on spaces of holomorphic functions on $\C^N$.
These operators are a composition of a differentiation operator and an affine composition operator, and
are analogues of operators studied by Aron and Markose on $H(\C)$.
The hypercyclic behavior is more involved than in the one dimensional case, and depends on several parameters involved.
\end{abstract}


\maketitle

%%%%%%%%%%%%%%%%%%%%%%%%%%%%%%%%%%%%%%%%%%%%%%%%%%%%%%%%%%%%
% Introduction

\section*{Introduction}
If $T$ is a continuous linear operator acting on some topological vector space $X$, the orbit under $T$ of a vector $x\in X$ is the set
$Orb(x,T):=\{x, Tx, T^2x,\dots\}$. The operator $T$ is said to be {\it hypercyclic } if there exists some vector $x\in X$, called {\it hypercyclic vector},
whose orbit under $T$ is dense in $X$. In the Fr\'echet space setting, an operator $T$ is hypercyclic if and
only if it is topologically transitive, that is, if for every pair
of non empty open sets $U$ and $V$, there exists a integer $n_0\in\N$ such $T^{n_0}U\cap V \neq \emptyset$. An operator
is said to be mixing if $T^nU\cap V \neq \emptyset$ for every integer $n\geq n_0$. Recently, some stronger
forms of hypercyclicity have gained the attention of researchers,
specially the concepts of frequently hypercyclic operators and strongly mixing operators with respect to some invariant probability measure on the space.

The first examples of hypercyclic operators were found by Birkhoff \cite{Bir29} and MacLane \cite{Mac52},
whose research was focused in holomorphic functions of one complex variable and not in properties of
operators. Birkhoff's result implies that the translation operator $\tau : H(\C)\to  H(\C)$ defined by $\tau(h)(z) =
h(1 + z)$ is hypercyclic. Likewise, MacLane's result states that the differentiation operator on $H(\C)$ is
hypercyclic. In a seminal paper, Godefroy and Shapiro \cite{GodSha91} unified and generalized both results,
by showing that every continuous linear operator $T : H(\C^N) \to H(\C^N)$ which commutes with translations
and which is not a multiple of the identity is hypercyclic. This operators are called non-trivial convolution operators.

Another important class of operators on $H(\C^N)$ are the composition operators $C_\phi$, induced by symbols $\phi$ which are automorphisms of $\C^N$. The hypercyclicity of composition operators induced by affine automorphisms was completely characterized in terms of properties of the symbol by Bernal-Gonz\'alez \cite{Ber05}.

Besides operators belonging to some of these two classes, there are not many examples of hypercyclic operators
on $H(\C^N)$.
Motivated by this fact, Aron and Markose \cite{AroMar04} studied the hypercyclicity of the following operator on $H(\mathbb C)$, $Tf(z)=f'(\lambda z+b)$,
with $\lambda, b\in\C$. The operator $T$ is not a convolution operator unless $\lambda= 1$.
%It is worth noting that is the composition of the operators studied earlier by Birkhoff and MacLane.
They showed that $T$ is hypercyclic for any $|\lambda|\geq 1$ (a gap in the proof was corrected in
\cite{FerHal05}) and that it is not hypercyclic if $|\lambda|<1$ and $b=0$. Thus, they gave explicit examples
of hypercyclic operators which are neither convolution operators nor composition operators. Recently, this
operators were studied in \cite{GupMun14}, where the authors showed that the operator is frequently
hypercyclic when $b=0$, $|\lambda|\ge1$ and asked whether it is frequently hypercyclic for any $b$.
In Section 2, we give a different proof of the result of \cite{AroMar04, FerHal05}, but for any
$\lambda,b\in\mathbb C$. We conclude in Proposition \ref{nonconvolution C} that $T$ is hypercyclic
if and only if $|\lambda|\ge1$, and that in this case, $T$ is
even strongly mixing with respect to some Borel probability measure of full support on $H(\C)$.

In Sections 3 and 4 we define $N$-dimensional analogues of the operators considered by Aron and Markose  and
study the dynamics they induce in $H(\C^N)$. These operators are a composition between a partial
differentiation operator and a composition operator induced by some automorphism of $\C^N$.
It turns out that its behavior is more complicated than its one variable analogue. One possible reason is
that, while the automorphisms of $\mathbb C$ have a very simple structure and hypercyclicity properties, the
automorphisms of $\mathbb C^N$ are much more involved. Even, the characterization of hypercyclic affine
automorphisms is nontrivial (see \cite{Ber05}).
% There are new cases in which these operators result  hypercyclic, or even  strongly mixing with respect to some fully supported probability Borel measure.

In Section 3, we consider the case in which the composition operators are induced by a diagonal operator plus a translation, that is, for $f\in H(\mathbb C^N)$ and $z=(z_1,\dots,z_N)\in\mathbb C^N$,
we study operators of the form $Tf(z)=D^\alpha f((\lambda_1z_1,\dots,\lambda_Nz_N)+b)$, where $\alpha$ is a multi-index and $b$ and $\lambda=(\lambda_1,\dots,\lambda_N)$ are vectors in $\mathbb C^N$.
In this case we completely characterize the hypercyclicity of these non-convolution operators which, contrary to the one dimensional case studied in \cite{AroMar04}, does not only depend on the size of $\lambda$.
In the last section, we study the operators which are a composition of a directional differentiation operator with a general affine automorphism of $\mathbb C^N$ and determine its hypercyclicity in some cases.

%%%%%%%%%%%%%%%%%%%%%%%%%%%%%%%%%%%%%%%%%%%%%%%%%%%%%%%%%%%%
% SECTION 1

\section{Preliminaries}
In this section we state some known conditions which ensure that a linear operator is strongly mixing with
respect to an invariant Borel probability
measure of full support. First we recall the following definitions.

\begin{definition}\rm
A linear operator $T$ on $X$ is called {\it frequently hypercyclic} if there exists a vector $x\in X$, called a {\it frequently hypercyclic vector}, whose $T$-orbit visit each non-empty open set along a set of integers having positive lower density.
\end{definition}

%\begin{theorem}
%Let $X$ be a complex, separable, Fr\'echet space and $T\in\mathcal{L}(X)$. Suppose that there exists a dense
%subspace $X_0\subset X$ such that $\sum_n T^n x$ is unconditionally convergent for all $x\in X_0$. Suppose
%further that for all $x\in X_0$ there exists a sequence $(u_k)\subset X$ such that $u_0=x$, $T^nu_k=u_{k-n}$
%if $k>n$ and $\sum_k u_k$ is unconditionally convergent. Then $T$ is frequently hypercyclic.
%\end{theorem}
%\noindent
%This theorem is in Grosse-Erdmann and Peris, Remark 9.10.

\begin{definition}\rm
A Borel probability measure on $X$ is Gaussian if and only if it is the
distribution of an almost surely convergent random series of the form $\xi=\sum_0^\infty{g_nx_n}$, where $(x_n) \subset X$ and $(g_n)$ is a sequence of independent, standard complex Gaussian variables.
\end{definition}

\begin{definition}\rm
We say that an operator $T\in \mathcal L(X)$ is strongly mixing in the Gaussian sense if there exists some Gaussian $T$-invariant probability measure $\mu$ on $X$ with full support such that any measurable sets $A,\,B\subset X$ satisfy
$$
\lim_{n\to\infty}\mu(A\cap T^{-n}(B))=\mu(A)\mu(B).
$$
\end{definition}

We will use the following result, which is a corollary of a theorem due to Bayart and Matheron (see
\cite{BayMatSMALL}). Essentially this theorem says that a large supply of eigenvectors associated to unimodular eigenvalues that are well distributed along the unit circle implies that the operator is strongly mixing in the Gaussian sense.
%Recall that the sets of
%extended uniqueness are those which are annihilated by every positive Rajchman measure in $\T$, and that
%Lebesgue measure is one of the family under consideration.

\begin{theorem}[Bayart, Matheron]\label{Bayart_Matheron}
Let $X$ be a complex separable Fr\'echet space, and let $T \in \mathcal L (X)$. Assume that for any set $D\subset \T$ such that $\T \setminus D$ is dense in $\T$, the linear span of
$\bigcup_{\lambda\in\T-D} \ker(T-\lambda)$ is dense in $X$. Then $T$ is strongly mixing in the Gaussian sense.
\end{theorem}

The following result, proved by Murillo-Arcila and Peris in \cite[Theorem 1]{MurPer13}, shows that operators
defined on Fr\'echet spaces which satisfy the Frequent Hypercyclicity Criterion are strongly mixing with respect to an
invariant Borel measure with full support.

\begin{theorem}[Murillo-Arcila, Peris]\label{murillo}
Let $X$ be a separable Fr\'echet space and $T\in\mathcal{L}(X)$. Suppose that there exists a dense subset
$X_0\subset X$ such that $\sum_n T^n x$ is unconditionally convergent for all $x\in X_0$. Suppose further that
there exists a sequence of maps $S_k:X_0\to X$ such that $T\circ S_1=Id$, $T\circ S_k=S_{k-1}$ and $\sum_k
S_k(x)$ is unconditionally convergent for all $x\in X_0$. Then there exists a Borel probability measure $\mu$
in $X$, $T$-invariant, such that the operator $T$ is strongly mixing respect to $\mu$.
\end{theorem}

It can be shown that the hypothesis of the Theorem \ref{murillo} implies the corresponding ones of the Theorem \ref{Bayart_Matheron}. So, in any case, both theorems allow us to conclude the existence of an invariant Gaussian probability measure for linear operators of full support which are strongly mixing.
Finally, the next proposition states that the existence of such measures is preserved by linear conjugation. It's proof is standard.

\begin{proposition}\label{conjugados}
Let $X$ and $Y$ be separable Fr\'echet spaces  and $T\in\mathcal{L}(X)$, $S\in\mathcal{L}(Y)$. Suppose that
$SJ=JT$ for some linear mapping $J:X\to Y$ of dense range then, if $T$ has an invariant Borel measure then so does $S$.
Moreover, if $T$ has an invariant Borel measure that is Gaussian, strongly mixing, ergodic or
of full support, then so does $S$.
\end{proposition}

%\end{proposition}
%satisfies the hypothesis of the previous theorem then $S$ also satisfies the hypothesis of the previous
%theorem.
%\begin{proof}
%Take $Y_0=J X_0$ dense in $Y$. If $y=Jx\in Y_0$, with $x\in X_0$ then $\sum_n S^ny$ is unconditionally
%convergent because, $\sum_n T^nx$ is unconditionally convergent and
%$$
%\sum_n S^ny=\sum_n S^nJx=\sum_n JT^n x=J\left(\sum_n T^nx\right).
%$$
%For $y=Jx\in Y_0$, define $U_k: Y_0\to Y$ as $U_ky=JS_k x$.
%%On other hand, if $(u_n)\subset X$ is the sequence associated to $x$, we take $v_n=J u_n\in Y_0$.
%It is clear that
%\begin{align*}
%&SU_1y=S J S_1 x= JTS_1 x =Jx=y,\\
%&SU_ky=S J S_k x= JTS_k x =J S_{k-1}x=U_{k-1}y,\\
%&\sum_kU_ky=\sum_k JS_kx =J\left(\sum_k S_k x\right) \textrm{is unconditionally convergent.}\\
%\end{align*}
%Therefore, $S$ satisfies the hypothesis of Theorem \ref{murillo}.
%\end{proof}

%%%%%%%%%%%%%%%%%%%%%%%%%%%%%%%%%%%%%%%%%%%%%%%%%%%%%%%%%%%%
% SECTION 2

\section{Non-Convolution operators on $H(\C)$}

%Recall Birkhoff's and MacLane's results about universal functions on $H(\C)$, the space of holomorphic
%functions of one complex variable.

Let us denote by $D$ and $\tau_a$ the derivation and translation operators on $H(\C)$, respectively. Namely, for an entire function $f$, we have
$$
D(f)(z)=f'(z) \;\; \text{ and } \;\; \tau_a(f)(z)=f(z+a).
$$
MacLane's theorem \cite{Mac52} says that $D$ is a hypercyclic operator, and
Birkhoff's theorem \cite{Bir29} states that $\tau_a$ is hypercyclic provided that $a\neq 0$.
The translation operators is a special class of composition operators on $H(\C)$. By a composition
operator we mean an operator $C_\phi$ such that $C_\phi(f)= f\circ \phi$, where $\phi$ is some automorphism of $\C$. The hypercyclicity of the composition operators on $H(\C)$ has been completely characterized in terms of
properties of the symbol function $\phi$. Precisely, the relevant property of $\phi$ is the following.

\begin{definition}\label{runaway}
A sequence $\{\phi_n\}_{n\in\N}$ of holomorphic maps on $\C$, is called {\it runaway} if, for each compact
set $K\subset \C$, there is an integer $n\in\N$ such that $\phi_n(K)\cap K= \emptyset$. In the case where
$\phi_n=\phi^n$ for every $n\in\mathbb N$, we will just say that $\phi$ is runaway.
\end{definition}

This definition was first given by Bernal Gonz\'alez and Montes-Rodr\'iguez in \cite{BerMon95}, where they
also proved the following (see also \cite[Therorem 4.32]{GroPer11}).
\begin{theorem}
Let $\phi$ be an automorphism of $\C$. Then $C_\phi$ is hypercyclic if and only if $\phi$ is
runaway.
\end{theorem}
It is known that the automorphisms of $\C$ are given by $\phi(z)=\lambda z + b$, with $\lambda\neq 0$ and
$b\in \C$. In addition, $\phi$ is runaway if and only if $\lambda=1$ and $b\neq 0$ (see \cite[Example
4.28]{GroPer11}). This means that the hypercyclic composition operators on $H(\C)$ are exactly Birkhoff's
translation operators.

Aron and Markose in \cite{AroMar04} studied the hypercyclicity of the following operator on $H(\C)$,
$$
Tf(z)=f'(\lambda z+b),
$$
with $\lambda, b\in\C$, which is a composition of MacLane's derivation operator and a composition
operator, i.e., $T= C_\phi \circ D$ with $\phi(z)=\lambda z + b$. The main motivation for the study of this operator was the wish to
understand the behavior of a concrete operator belonging neither to the class of convolution operators nor to
the class of composition operators.
As mentioned before, in \cite{AroMar04} (see also \cite{FerHal05}) the authors proved that $T$ is hypercyclic
if $|\lambda|\geq 1$, and that it is not hypercyclic if $|\lambda|<1$ and $b=0$.

In this section we give a simple proof of the result by Aron and Markose,
for the full range on $\lambda,b$. This will allow us to illustrate some of the main ideas used in
the next section to prove the more involved $N$-variables case.
% Note that if $\lambda=1$ then $T$ is a convolution
% operator, and therefore it is a hypercyclic operator  by the classical result of Godefroy and Shapiro
% \cite{GodSha91}.

Suppose that $\lambda\neq 1$. The key observation is that $T$ is conjugate to an operator of the same type,
but with $b=0$. Indeed, define $T_0f(z)=f'(\lambda z)$, then
% since
% $$f(z)
% \overset{\tau_{[\frac{b}{1-\lambda}]}}{\longmapsto}f\left(z-\frac{b}{\lambda-1}\right)\overset{T_0}
% \longmapsto f'\left(\lambda z-\frac{b}{\lambda-1}\right)
%     \overset{\tau_{[\frac{b}{\lambda-1}]}}\longmapsto
% f'\left(\lambda\left(z+\frac{b}{\lambda-1}\right)-\frac{b}{\lambda-1}\right)=Tf(z).
% $$
we have that the following diagram commutes.
$$
    \xymatrix{H(\C) \ar[r]^{T} \ar[d]_{\tau_{[\frac{b}{1-\lambda}]}} & H(\C)  \\ H(\C) \ar[r]_{T_0} & H(\C)
\ar[u]_{\tau_{[\frac{b}{\lambda-1}]}} }
    $$

Note that $\frac{b}{1-\lambda}$ is the fixed point of $\phi$. This observation will be important later.

\begin{proposition}\label{nonconvolution C}
Let $T$ be the operator defined on $H(\C)$ by $Tf(z)=f'(\lambda z+b)$. Then $T$ is hypercyclic if and only if $|\lambda|\geq 1$. In this case, $T$ is
also strongly mixing with respect to some Borel probability measure of full support on $H(\C)$.
\end{proposition}

\begin{proof}
If $\lambda=1$, then $T$ is a non-trivial convolution operator, thus it is hypercyclic. Moreover, by the Godefroy and Shapiro's theorem and its extensions (see \cite{GodSha91,BonGro06,MurPinSav14}), $T$ is strongly mixing in the Gaussian sense. Hence, by Proposition \ref{conjugados}, it suffices to prove the case
$b=0$ and $\lambda\ne1$, i.e. for the operator $T_0$.

Suppose first that $|\lambda|<1$ and let $f\in H(\C)$. Note
that $T_0^nf(z)=\lambda^{\frac{n(n-1)}{2}}f^{(n)}(\lambda^n z).$
By the Cauchy's estimates we obtain that
$$
|T_0^nf(0)|\leq |\lambda|^{\frac{n(n-1)}{2}}n!\sup_{\|z\|\leq
1}|f(z)| \mathop{\longrightarrow}\limits_{n\to\infty}0.
$$
Since the evaluation at $0$ is continuous, the orbit of $f$ under $T_0$ can not be dense.

Suppose now that $|\lambda|> 1$. Let us see that we can apply the Murillo-Arcila and Peris criterion, Theorem
\ref{murillo}. Let $X_0$ be the set of all polynomials, which is dense in $H(\C)$. Then, for each polynomial
$f\in X_0$, the series $\sum_n T_0^nf$ is actually a finite sum, thus it is unconditionally convergent.

For $n\in\N$ we define a sequence of linear maps $S_n:X_0\to X$ as
$$S_n(z^k)=\frac{k!}{(k+n)!}\frac{z^{k+n}}{\lambda^{nk+\frac{n(n-1)}{2}}}.$$
It is easy to see that $S_n$ satisfy the hypothesis of Theorem \ref{murillo}.
\begin{itemize}
\item $T_0\circ S_1=I:$
$$T_0\circ S_1(z^k)=T_0\left(\frac{1}{k+1}\frac{z^{k+1}}{\lambda^{k}}\right)=z^k. $$
\item $T_0\circ S_n=S_{n-1}:$
\begin{align*}T_0\circ S_n(z^k)&=T_0\left(\frac{k!}{(k+n)!}\frac{z^{k+n}}{\lambda^{nk+\frac{n(n-1)}{2}}}\right)\\
    &=\frac{k!}{(k+n-1)!}\frac{\lambda^{k+n-1}z^{k+n-1}}{\lambda^{nk+\frac{n(n-1)}{2}}}\\
     &=\frac{k!}{(k+n-1)!}\frac{z^{k+n-1}}{\lambda^{(n-1)k+\frac{(n-1)(n-2)}{2}}}\\
     &=S_{n-1}(z^k).
\end{align*}
\item The series $\sum_n S_n(f)$ is unconditionally convergent for each $f\in X_0$. If $|z|\leq R$, we get that,
$$\sum_n|S_n(z^k)|\leq
\sum_n \frac{k!}{(k+n)!}R^{k+n}\leq k! e^R.$$
\end{itemize}
Thus, the operator $T_0$ is strongly mixing in the Gaussian sense.
\end{proof}

We can summarize the results of this section in the following table. It is worth noticing that nor the
hypercyclicity of $C_\phi$ nor the hypercyclicity of $D$ imply the hypercyclicity of $C_\phi\circ D$.

\bigskip

\begin{center}
\begin{tabular}{l|c|c|c}
                 & $\lambda <1$    & $\lambda =1$                            & $\lambda > 1$\\ \hline
$C_\phi$         & Not Hypercyclic & Hypercyclic $\Leftrightarrow$ $b\neq 0$ & Not Hypercyclic\\ \hline
$D$              & Hypercyclic     & Hypercyclic                             & Hypercyclic\\ \hline
$C_\phi \circ D$ & Not Hypercyclic & Hypercyclic                             & Hypercyclic\\
\end{tabular}
\end{center}

%%%%%%%%%%%%%%%%%%%%%%%%%%%%%%%%%%%%%%%%%%%%%%%%%%%%%%%%%%%%
%SECTION 3

\section{Non-Convolution operators on $H(\C^N)$ - the diagonal case}

The operators considered in the previous section were differentiation operators followed by a composition operator.
In this section we consider $N$-dimensional analogues of those operators. First, we will be concerned with symbols $\phi:\C^N\to\C^N,$ which are diagonal affine automorphism of the form
$$
\phi(z)=\lambda z+b = (\lambda_1 z_1+b_1,\dots, \lambda_N z_N+b_N),
$$
where $\lambda, b\in\C^N$; and the differentiation operator is a partial derivative operator given by a multi-index $\alpha=(\alpha_1\dots,\alpha_N) \in \N_0^N$,
$$
D^\alpha f=\frac{\partial^{|\alpha|}f}{\partial z_1^{\alpha_1}\partial z_2^{\alpha_2}\dots \partial
z_N^{\alpha_N}}.
$$
Thus in this section $T$ will denote the operator on $H(\C^N)$ defined by
$$Tf(z)=C_\phi\circ D^\alpha(f)(z) = D^\alpha f(\lambda_1 z_1+b_1,\dots, \lambda_N z_N+b_N).$$
Note that, in the definition of $T$, we allow $\alpha$ to be zero. In this case, the operator is just a
composition  operator and its hypercyclicity is determined by the symbol $\phi$.
% , just like in the case of $H(\C)$.
These symbol functions are special cases of affine automorphisms of $\C^N$. The existence of universal functions for composition operators with affine symbol on $\C^N$ has been completely characterized by Bernal-Gonzalez in \cite{Ber05}, where he proved that the hypercyclicity of the composition operator depends on whether or not the symbol is runaway. Recall that an automorphism $\varphi$ of $\mathbb C^N$ is said to be runaway if for any compact subset $K$ there is some $n\ge1$ such that $\varphi^n(K)\cap K=\emptyset$.

\begin{theorem}[Bernal-Gonz\'alez]\label{phi runaway}
Assume that $\varphi :\C^N \to\C^N$ is an affine automorphism of $\C^N$, say $\varphi(z) = Az+b$. Then, the composition operator $C_\varphi$ is hypercyclic if and only if $\varphi$ is a runaway automorphism if and only if  the vector $b$ is not in $ran(A -I)$ and $det(A) \neq 0$.
\end{theorem}
The proof of this result is based on the following $N$-variables generalization of Runge's approximation theorem, which will be useful for us later.
\begin{theorem}\label{Runge}
If $K$ and $L$ are disjoint convex compact sets in $\C^N$ and $f$ is a holomorphic function in a neighborhood of $K\cup L$, then there is a sequence of polynomials on $\C^N$ that approximate $f$ uniformly on $K\cup L$.
\end{theorem}

\begin{remark}\rm
It is easy to prove that the mapping $\phi(z)=(\lambda_1z_1+b_1,\dots,\lambda_Nz_N+b_N)$ is runaway if and only if
%$\lambda_j\neq 0$ for all $j=1,\dots,N$ and
some coordinate is a translation, that is, for some $i=1,\dots,N$ we have, simultaneously, that $\lambda_i=1$ and $b_i\neq 0$.
% of the form $z+b$ with $b\neq 0$.
%Indeed, suppose that $\lambda_i\neq 1$ for all $i$. Then, $\left(\frac{b_1}{1-\lambda_1},\dots,\frac{b_N}{1-\lambda_N}\right)$ is a fix point of $\phi$, hence is not runaway. Denote $\phi_j(z)=\lambda_j z+b_j$ and suppose that there is some $\lambda_i=1$. Then, we have that $\phi_i^n(z)=z+nb_i$, which is runaway in $\C$ if and only if $b_i\neq 0$. So, $\phi$ is runaway in $\C^N$ if an only if there is some coordinate with $\lambda_i=1$ and $b_i\neq 0$.
\end{remark}

If $\lambda_j=0$ for some $j$, then we have that the differential $d(T^nf)(e_j)=1$, for every $n \in \zN$.
Since, the application $d(\cdot)(e_j)$ is continuous, we conclude that the orbit of $f$ under $T$ can not be
dense.

% some $\lambda_j$ is 0, we have that $\frac{d^k((e'_j)^k-T^nf)}{k!}(e_j) =1$, for each $n \in \zN$. Since, the application $d^k(\cdot)(e_j)$ is continuous, we get that the orbit of $f$ under $T$ can not be dense.

%for example $\lambda_1=0$, then $Tf(z)=D^\alpha f(0, \lambda_2z_2+b_2, \ldots, \lambda_N z_N+b_N)$ for every $f \in H(\C^N)$,
%$\left\|\frac{d^k((e'_1)^k-T^nf)}{k!}\right\|_\infty\geq 1$

The next result completely characterizes  the hypercyclicity of the operator $Tf=C_\phi\circ D^\alpha f$, with
$\lambda \neq 0$ and $\alpha\neq 0$ (the case $\alpha=0$ is covered in \cite{Ber05}, and as mentioned above
$T$ is not hypercyclic if $\lambda_j=0$ for some $j$). Write $\lambda^\alpha=\prod_{i\leq N}
\lambda_i^{\alpha_i}$.

\begin{theorem}\label{nonconvolution CN}
  Let $T$ be the operator on $H(\C^N)$, defined by $Tf(z)=C_\phi\circ D^\alpha f(z)$, where $\alpha\neq 0$, $\phi(z)=(\lambda_1z_1+b_1,\dots,\lambda_Nz_N+b_N)$ and $\lambda_i\neq 0$ for all $i$, $1\le i\le N$. Then,

  $a)$ If $|\lambda^\alpha|\ge1$ then $T$ is strongly mixing in the Gaussian sense.

  $b)$ If for some $i=1,\dots,N$ we have that $b_i\ne0$ and $\lambda_i=1$, then $T$ is mixing.

  $c)$ In any other case, $T$ is not hypercyclic.
\end{theorem}
% \begin{theorem}\label{nonconvolution CN}
%   Let $T$ be the operator on $H(\C^N)$, defined by $Tf(z)=C_\phi\circ D^\alpha f(z)$, with $\alpha\neq 0$, $\phi(z)=(\lambda_1z_1+b_1,\dots,\lambda_Nz_N+b_N)$ and $\lambda_i\neq 0$ for all $i$, $1\le i\le N$. Then,
%
%   $a)$ If $|\lambda^\alpha|\ge1$ then $T$ is strongly mixing with respect to some Borel probability measure with full support.
%
%   $b)$ If $|\lambda^\alpha|<1$ and $b=0$ then $T$ is not hypercyclic.
%
%   $c)$ If $|\lambda^\alpha|<1$ and $\lambda_i\ne1$ for every $i$, $1\le i\le N$ then $T$ is not hypercyclic.
%
%   $d)$ If $|\lambda^\alpha|<1$ and $b_i=0$ for every $i$ such that $\lambda_i=1$ then $T$ is not hypercyclic.
%
%   $e)$ If $b_i\ne0$ for some $i$ such that $\lambda_i=1$ then $T$ is mixing.
% \end{theorem}

\begin{remark}\rm
The item \textit{c)} above includes the following cases:

$c-i)$  $|\lambda^\alpha|<1$ and $b=0$.

$c-ii)$ $|\lambda^\alpha|<1$ and $\lambda_i\ne1$ for every $i$, $1\le i\le N$.

$c-iii)$ $|\lambda^\alpha|<1$ and $b_i=0$ for every $i$ such that $\lambda_i=1$.

In all three cases we have that the application $\phi(z)=\lambda z+b$ has a fixed point and thus $\phi$ is not
runaway. Also in case $b)$ the application $\phi$ has one coordinate which is a
translation, thus it is runaway. So, in particular, Theorem \ref{nonconvolution CN} implies that $T=C_\phi\circ
D^\alpha$ is hypercyclic if and only if either $|\lambda^\alpha| \geq1$ or $\phi$ is runaway.
\end{remark}

We can summarize our main theorem in the following table.

\bigskip

\begin{center}
\begin{tabular}{l|c|c|c}
                 & $|\lambda^\alpha| <1$ and & $|\lambda^\alpha| <1$ and & $|\lambda^\alpha| \geq1$  \\
								 & no coord. of $\phi$ is a translation & a coord. of $\phi$ is a translation    & \\ \hline
$C_\phi$         & Not Hypercyclic & Hypercyclic & depends on $\phi$\\ \hline
$D^\alpha$              & Hypercyclic     & Hypercyclic                      & Hypercyclic\\ \hline
$C_\phi \circ D^\alpha$ & Not Hypercyclic & Hypercyclic                & Hypercyclic\\
\end{tabular}
\end{center}

\bigskip

We will divide the proof of part (a) of Theorem \ref{nonconvolution CN} in two lemmas. Through a change in the
order of the variables, we may suppose that the first $j$ variables, $0\le j\le N$, correspond to the
coordinates in which $\lambda_i=1$. The operator $T$ is then of the form
\begin{equation}\label{algunos 1}
Tf(z)=D^\alpha f(z_1+b_1,\dots, z_j+b_j, \lambda_{j+1}z_{j+1}+b_{j+1}, \dots, \lambda_Nz_N+b_N).
\end{equation}
Moreover, we can assume that $b_i=0$ for all $i>j$, because $T$ is topologically conjugate to
\begin{equation}\label{algunos 1 ya trasladado}
T_0f(z)=D^\alpha f(z_1+b_1,\dots, z_j+b_j, \lambda_{j+1}z_{j+1}, \dots, \lambda_Nz_N).
\end{equation}
through a translation. Indeed,  defining $c\in\C^N$ by
$c_l=0$ if $l\leq j,$ and $c_l=\frac{b_l}{1-\lambda_l}$ if $l>j$, we get that $T_0\circ\tau_{c}=\tau_{c}\circ T$.

We first study the case in which for some $i$, we have $\lambda_i\neq 1$ and $\alpha_i\neq 0$ (note that if all $\lambda_i=1$, then $T$ is a convolution operator and it is thus strongly mixing in the Gaussian sense \cite{BonGro06,MurPinSav14}).

%%We claim that $T$ is frequently hypercyclic if and only if $|\lambda^\alpha|\geq1$. We will divide the proof
%in two cases. Suppose first that $\alpha$ has none-zero coordinates corresponding to any variable of index
%higher than $j$.

\begin{lemma}\label{nonconvolution CN lema1}
Let $T$ be as in (\ref{algunos 1}). Suppose that $|\lambda^\alpha|\geq 1$ and $\alpha_i\neq0$ for some $i>j$. Then $T$ is strongly mixing in the Gaussian sense.
\end{lemma}
\begin{proof}
By the above comments, we may suppose that $b_i=0$ for $i>j$, so the operator $T$ is as in (\ref{algunos 1 ya trasladado}).
We apply Theorem \ref{murillo} with
$$X_0=span\left\{e_\gamma z^\beta:=e^{\gamma_1z_1+\dots+\gamma_jz_j}z^\beta \textrm{ with } \beta_i=0 \textrm{
for } i\leq j \textrm{ and } \gamma\in \C^j\right\}.$$
The set $X_0\subset H(\C^N)$ is dense. Indeed, since the set $\{e_\gamma: \gamma\in \C^j\}$ generates a dense
subspace in $H(\C^j)$ (see for example \cite[Proposition 2.4]{BonGro06}), given a monomial $z_1^{\theta_1}\dots z_j^{\theta_j}$,
$\epsilon>0$ and $R>0$, there is $f\in span\{e_\gamma:\gamma\in \C^j\}$ with

$$
\sup_{\|z\|\leq R}\left|f(z_1,\dots,z_j)-z_1^{\theta_1}\dots z_j^{\theta_j}\right|<\epsilon.
$$
We obtain
$$
\sup_{\|z\|\leq R}\left|\underbrace{f(z_1,\dots,z_j)z_{j+1}^{\beta_{j+1}}\dots z_{N}^{\beta_{N}}}_{\in
X_0}-z_1^{\theta_1}\dots z_j^{\theta_j}z_{j+1}^{\beta_{j+1}}\dots z_{N}^{\beta_{N}}\right|<\epsilon R^{|\beta|}.
$$
Therefore we can approximate any monomial in $H(\C^N)$ by functions of $X_0$ uniformly on compacts sets.

The series $\sum_n T^n(e_\gamma z^\beta)$ is unconditionally convergent because the operator $T$
differentiates in some variable $z_i$ with $i>j$, and so it is a finite sum. On the other hand, if we denote by
$\alpha_{(1)}:=(\alpha_1,\dots,\alpha_j)$ and $\alpha_{(2)}:=(\alpha_{j+1},\dots,\alpha_N)\neq 0$, we obtain
$$T^n(e_\gamma
z^\beta)=\gamma^{n\alpha_{(1)}}e^{n\langle\gamma,b\rangle}\lambda^{n\beta-\frac{n(n+1)}{2}\alpha_{(2)}}\frac{
\beta!}{(\beta-n\alpha_{(2)})!}e_\gamma z^{\beta-n\alpha_{(2)}}.$$

Now, we define a sequence of maps $S_n:X_0\to X_0$. First, we do that on the set $\{e_\gamma z^\beta\}$ and then extending them by
linearity
$$
S_n(e_\gamma
z^\beta)=\frac{\beta!}{\gamma^{n\alpha_{(1)}}e^{n\langle\gamma,b\rangle}\lambda^{n\beta+\frac{n(n-1)}{2}
\alpha_{(2)}}(\beta+n\alpha_{(2)})!}e_\gamma z^{\beta+n\alpha_{(2)}}.
$$

The following assertions hold:
\begin{itemize}
\item $T\circ S_1=I:$ \begin{align*}T\circ S_1(e_\gamma
z^\beta)&=\frac{1}{\gamma^{\alpha_{(1)}}e^{\langle\gamma,b\rangle}\lambda^{\beta}}\frac{\beta!}{(\beta+\alpha_
{(2)})!}T(e_\gamma z^{\beta+\alpha_{(2)}})\\
&=\frac{1}{\gamma^{\alpha_{(1)}}e^{\langle\gamma,b\rangle}\lambda^{\beta}}\frac{\beta!}{(\beta+\alpha_{(2)})!}
\gamma^{\alpha_{(1)}}e^{\langle\gamma,b\rangle}e_{\gamma}\frac{(\beta+\alpha_{(2)})!}{\beta!}
z^\beta\lambda^\beta\\&=e_\gamma z^\beta.  \end{align*}
\item $T\circ S_n=S_{n-1}:$ \begin{align*}T&\circ S_n(e_\gamma
z^\beta)=\frac{1}{\gamma^{n\alpha_{(1)}}e^{n\langle\gamma,b\rangle}\lambda^{n\beta+\frac{n(n-1)}{2}\alpha_{(2)
}}}\frac{\beta!}{(\beta+n\alpha_{(2)})!}T(e_\gamma z^{\beta+n\alpha_{(2)}})\\
&=\frac{\beta!\gamma^{\alpha_{(1)}}e^{\langle\gamma,b\rangle}\lambda^{\beta+(n-1)\alpha_{(2)}}(\beta+n\alpha_{
(2)})!}{\gamma^{n\alpha_{(1)}}e^{n\langle\gamma,b\rangle}\lambda^{n\beta+\frac{n(n-1)}{2}\alpha_{(2)}}
(\beta+n\alpha_{(2)})!(\beta+(n-1)\alpha_{(2)})!}e_\gamma z^{\beta+(n-1)\alpha_{(2)}}\\&=
\frac{\beta!}{\gamma^{(n-1)\alpha_{(1)}}e^{(n-1)\langle\gamma,b\rangle}\lambda^{(n-1)\beta+\frac{(n-1)(n-2)}{2
}\alpha_{(2)}}(\beta+(n-1)\alpha_{(2)})!}e_\gamma z^{\beta+(n-1)\alpha_{(2)}}\\&= S_{n-1}(e_\gamma z^\beta).
     \end{align*}
\item Given $R>0$, let $|z|\le R$ and denote $C=|\frac{R^{\alpha_{(2)}}}{\lambda^\beta
\gamma^{\alpha_{(1)}}e^{\langle\gamma,b\rangle}}|$. We have $|S_n(e_\gamma z^\beta)|\leq M
\frac{C^n}{(\beta+n\alpha_{(2)})!}$ for some constant $M>0$ not depending on $n$. Since, $\alpha_{(2)}\neq 0$, we get that for each $\gamma\in\C^j$ and $\beta\in\C^N$ with $\beta_i=0$ for $i\leq j$,
$\sum_n |S_n(e_\gamma z^\beta)|$ is uniformly convergent on compacts sets.
\end{itemize}
We have thus shown that the hypothesis of Theorem \ref{murillo} are fulfilled. Hence $T$ is strongly mixing in the Gaussian sense, as we wanted to prove.
\end{proof}

The other case we need to prove is when $T$ does not differentiate in the variables $z_i$ with $i>j$. This means that $\alpha_i = 0 $ for all $i>j$. To prove this case we will use Theorem \ref{Bayart_Matheron}.
\begin{lemma}\label{nonconvolution CN lema2}
Let $T$ be as in (\ref{algunos 1}). Suppose that $|\lambda^\alpha|\geq 1$ and $\alpha_i=0$ for every $i>j$.
Then $T$ is strongly mixing in the Gaussian sense.
\end{lemma}

\begin{proof}
We may suppose that $b_i=0$ for $i>j$, so the operator $T$ is as in (\ref{algunos 1 ya trasladado}).
The functions $e_\gamma z^\beta$, with $\gamma_i=0$ for all $i>j$ and $\beta_i=0$ for every $i\le j$, are eigenfunctions of $T$. Indeed,
$$
T(e_\gamma z^\beta) = \gamma^{\alpha_{(1)}} e^{\sum \gamma_i(z_i+b_i)} (\lambda z)^\beta
=\gamma^{\alpha_{(1)}}\lambda^\beta e^{\langle \gamma,b \rangle} e_\gamma
z^\beta,$$
where, as in the proof of the last lemma, $\alpha_{(1)}=(\alpha_1,\dots,\alpha_j)\neq 0$ (note that in this
case $\alpha_{(2)}=(\alpha_{j+1},\dots,\alpha_N)=0$).

By Theorem \ref{Bayart_Matheron} it is enough to show that for every set $D\subset\T$ such that $\T\setminus D$ is dense in $\T$, the set
%$$\left\{e_\gamma z^\beta; \, \beta\in\C^N\textrm{ with } \beta_i=0 \textrm{ for } i>j \textrm{ and } \gamma
%\textrm{ such that } \gamma^\alpha \lambda^\beta e^{\langle \gamma,b \rangle}\in \T\setminus D \right\}$$
\begin{equation}\label{denso lema2}
\left\{e_\gamma z^\beta; \, \beta\in\C^N\textrm{ with } \beta_i=0 \textrm{ for } i\le j \textrm{ and } \gamma_i=0 \textrm{ for } i> j,
\textrm{ such that } \gamma^\alpha \lambda^\beta e^{\langle \gamma,b \rangle}\in \T\setminus D \right\},
\end{equation}
spans a dense subspace on $H(\C^N)$.

Fix $\beta\in\C^N$ with $\beta_i=0$ for every $i\le j$ and consider the map
\vspace{-1mm}
\begin{align*}
f_\beta:\C^j&\to \C\\
\gamma &\mapsto \gamma^\alpha \lambda^\beta e^{\langle \gamma,b \rangle}.
\end{align*}
\vspace{-1mm}
The application $f_\beta$ is holomorphic and non constant. So there exists $\gamma_0\in\C^j$ such that $|{\gamma_0}^\alpha \lambda^\beta
e^{\langle \gamma_0,b \rangle}|=1$. Since, $\T\setminus D$ is a dense set in $\T$, the vector $\gamma_0$ is an accumulation point of $\T\setminus D$. Thus, by \cite[Proposition 2.4]{BonGro06}, we get that the set $$\left\{e_\gamma; \,\textrm{with }\gamma
\textrm{ such that } \gamma^\alpha \lambda^\beta e^{\langle \gamma,b \rangle}\in\T\setminus D\right\},$$ spans a dense subspace in $H(\C^j)$. It is then easy to see that the set defined in \eqref{denso lema2} spans a dense subspace in $H(\C^N)$.
%\textcolor{red}{demasiado poca explicacion? me parecia que explicar de mas confundia...}
In particular, we have shown that the set of eigenvectors of $T$ associated to eigenvalues belonging to $\T\setminus D$ span a dense subspace in $H(\C^N)$. So, the hypothesis of Theorem \ref{Bayart_Matheron} are satisfied and hence $T$ is strongly mixing in the Gaussian sense.
\end{proof}

The following remark will be useful for the next proof and in the rest of the article.
\begin{remark}\rm
Recall the Cauchy's formula for holomorphic functions in $\C^N$,
$$
D^\alpha f(z_1,\dots,z_N)=\frac{\alpha!}{(2\pi i)^N} \int_{|w_1-z_1|=r_1}\dots \int_{|w_N-z_N|=r_N} \frac{f(w_1,\dots,w_N)}{\prod_{i=1}^N(w_i-z_i)^{\alpha_i+1}}  \,dw_1\dots dw_N.
$$
Therefore, we can estimate the supremum  of $D^\alpha f$ over a set of the form $B(z_1,r_1)\times\dots\times
B(z_N,r_N)$, where $B(z_j,r_j)$ denotes the closed disk of center $z_j\in\mathbb C$ and radius $r_j$. Fix
positive real numbers $\varepsilon_1,\dots,\varepsilon_N$, then
\begin{equation}\label{cauchy}
\|D^\alpha f\|_{\infty, B(z_1,r_1)\times\dots\times B(z_N,r_N)}\leq \frac{\alpha!}{(2\pi)^N} \frac{\|f\|_{\infty,B(z_1,r_1+\varepsilon_1)\times\dots\times B(z_N,r_N+\varepsilon_N)}}{\varepsilon_1^{\alpha_1+1}\dots\varepsilon_N^{\alpha_N+1}}.
\end{equation}
\end{remark}

\begin{proof}(of Theorem \ref{nonconvolution CN})
Part $a)$ is proved by Lemmas \ref{nonconvolution CN lema1} and \ref{nonconvolution CN lema2}.

$b)$ Suppose that $b_l\ne0$ for some $l$ such that $\lambda_l=1$. We will prove that $T$ is a mixing operator,
i.e., that for every pair  $U$ and $V$ of non empty open sets for the local uniform topology of $H(\C^N)$,
there exists $n_0\in\N$ such that $T^n(U)\cap V\neq\emptyset$ for all $n\geq n_0$. Let $f$ and $g$ be two
holomorphic functions on $H(\C^N)$, $L$ be a compact set of $\C^N$ and $\theta$ a positive real number. We can
assume that
\[
U=\{h\in H(\C^N)\,: \; \|f-h\|_{\infty,L}<\theta \} \; \text{ and } \; V=\{h\in H(\C^N)\,: \; \|g-h\|_{\infty,L}<\theta \},
\]
and that $g$ is a polynomial and that $L$ is a closed ball  of $(\C^N,\|\cdot\|_{\infty})$. We do so because we can define a right inverse map over the set of polynomials.
Since $T=C_{\phi}\circ D^{\alpha}$, we can define
$$
I^\alpha(z^\beta)=\frac{\beta!}{(\alpha+\beta)!}z^{\alpha+\beta}.
$$
Thus, $S=I^\alpha\circ C_{\phi^{-1}}$ is a right inverse for $T$ when restricted to polynomials.
Hence, we assume that $L=B(0,r)\times B(0,r)\times\dots\times B(0,r)$, for some $r>0$ and denote $\phi_i(z)=\lambda_i z+b_i$, for $z\in\C$. We get that $\phi(z_1,\dots,z_N)=(\phi_1(z_1)\dots,\phi_N(z_N))$ and $\phi_i(B(z_i,r_i))
=B(\phi_i(z_i),|\lambda_i|r_i)$.

Now, suppose that $P$ is a polynomial in $\C^N$. Applying the inequality (\ref{cauchy}) several times, in which each time we use it we divide each $\varepsilon_i$ by $2$,  we get that

\begin{align*}
\left\|g-T^nP\right\|_{\infty, L}&=  \left\|C_{\phi}\circ D^{\alpha}(Sg-T^{n-1}P)\right\|_{\infty, L}=
\left\| D^{\alpha}(Sg-T^{n-1}P)\right\|_{\infty, \phi(L)}
\\&=  \left\| D^{\alpha}(Sg-T^{n-1}P)\right\|_{\infty, \prod B(b_i,|\lambda_i|r)}
\\& \leq \frac{\alpha!}{ {(2\pi)}^N  \varepsilon_1^{\alpha_1+1}\dots\varepsilon_N^{\alpha_N+1}} \left\| Sg-T^{n-1}P\right\|_{\infty, \prod B(b_i,|\lambda_i|r+\varepsilon_i)}
\\& \leq \frac{\alpha!}{ {(2\pi)}^N  \varepsilon_1^{\alpha_1+1}\dots\varepsilon_N^{\alpha_N+1}} \left\| C_{\phi}\circ D^\alpha(S^2g-T^{n-2}P)\right\|_{\infty,\prod B(b_i,|\lambda_i|r+\varepsilon_i)}
\\& \leq \frac{\alpha!}{ {(2\pi)}^N  \varepsilon_1^{\alpha_1+1}\dots\varepsilon_N^{\alpha_N+1}} \left\| D^\alpha(S^2g-T^{n-2}P)\right\|_{\infty,\prod B((\lambda_i+1)b_i,|\lambda_i|(|\lambda_i|r+\varepsilon_i))}
 \\& \leq \frac{2^{|\alpha|+N}\alpha!^2}{ {(2\pi)}^{2N}  \varepsilon_1^{2(\alpha_1+1)}\dots\varepsilon_N^{2(\alpha_N+1)}} \left\| S^2g-T^{n-2}P\right\|_{\infty,\prod B((\lambda_i+1)b_i,|\lambda_i|(|\lambda_i|r+\varepsilon_i)+\frac{\varepsilon_i}{2})}
\end{align*}
Thus following, we get that
\begin{align*}
\left\|g-T^nP\right\|_{\infty, L}\leq \frac{2^{(n(n+1)/2)(|\alpha|+N)}\alpha!^n}{ {(2\pi)}^{nN}  \varepsilon_1^{n(\alpha_1+1)}\dots\varepsilon_N^{n(\alpha_N+1)}} \left\| S^ng-P\right\|_{\infty,\prod B\left(\phi_i^n(0),|\lambda_i|^nr+\varepsilon_i\sum_{k=0}^{n-1}\frac{|\lambda_i|^k}{2^{n-k-1}}\right)}.
\end{align*}

Let us denote by $l$, the coordinate of $\phi$ that is a translation in $\C$. Thus, we have that $\lambda_l=1$ and $b_l\neq 0$. This implies that
$$
B\left(\phi_l^n(0),|\lambda_l|^nr+\varepsilon_l\sum_{k=0}^{n-1}\frac{|\lambda_l|^k}{2^{n-k-1}}\right)= B\left(nb_l,r+\varepsilon_l\sum_{k=0}^{n-1}\frac{1}{2^k}\right)\subset B\left(nb_l, r+2\varepsilon_l\right).
$$
Fix $n_0\in \N$, such that $B(0,r)\cap B\left(nb_l, r+2\varepsilon_l\right)=\emptyset$ for all $n\geq n_0$.
Now, take $\delta_n>0$ and $\Lambda_n$ a ball of $(\C^N,\|\cdot\|_{\infty})$, such that $[L+\delta_n]\cap[\Lambda_n+\delta_n]=\emptyset$ for all $n\geq n_0$ and
$$\prod_{i=1}^NB\left(\phi_l^n(0),|\lambda_l|^nr+\varepsilon_l\sum_{k=0}^{n-1}\frac{|\lambda_l|^k}{2^{n-k-1}}\right)\subset \Lambda_n.$$
Also, denote by  $K_n=\frac{2^{(n(n+1)/2)(|\alpha|+N)}\alpha!^n}{ {(2\pi)}^{nN}  \varepsilon_1^{n(\alpha_1+1)}\dots\varepsilon_N^{n(\alpha_N+1)}}$. Then, use Theorem \ref{Runge} with $h_n=\chi_{L+\delta_n}f+\chi_{\Lambda_n+\delta_n}S^ng$.
We get a polynomial $P_n$ such that
$$
\|f-P_n\|_L<\theta \; \text{ and } \; \|S^ng-P_n\|_{\Lambda_n}<\frac{\theta}{K_n}.
$$
Hence,
$$
\|f-P_n\|_L<\theta \; \text{ and } \; \|g-T^nP_n\|_{L}<\theta.
$$
Thus, $P_n\in U\cap T^{-n}V$ for all $n\geq n_0$ and $T$ is a mixing operator as we wanted to prove.

$c)$ Let $\frac{b}{1-\lambda}=(\frac{b_1}{1-\lambda_1},\dots,\frac{b_N}{1-\lambda_N})$ where, if $b_j=0$ and
$\lambda_j=0$ for some $j=1,\dots,N$,  we will understand that $\frac{b_j}{1-\lambda_j}=0$. Then
$\frac{b}{1-\lambda}$ is a fixed point of $\phi$, and thus
$$
T^nf\left(\frac{b}{1-\lambda}\right)=\lambda^{\frac{n(n-1)}{2}\alpha}D^{n\alpha}f\left(\frac{b}{1-\lambda}
\right).
$$
Applying the Cauchy estimates we obtain
$$
\left|T^n f\left(\frac{b}{1-\lambda}\right)\right|\leq
|\lambda^\alpha|^{\frac{n(n-1)}{2}}\left|D^{n\alpha}f\left(\frac{b}{1-\lambda}\right)\right|\leq
\frac{|\lambda^\alpha|^{\frac{n(n-1)}{2}}(n\alpha)!}{r^{n|\alpha|}}\sup_{\|z\|\leq
r}|f(z)|\mathop{\longrightarrow}_{n\to\infty}0.
$$
Since the evaluation at the vector $\frac{b}{1-\lambda}$ is a continuous
functional, this implies that the orbit of $f$ under $T$ is not dense.
\end{proof}

Notice that in case $b)$ of Theorem \ref{nonconvolution CN} we do not know if the operator $C_\phi \circ D^\alpha$ is strongly mixing in the Gaussian sense or even frequently hypercyclic. If $|\lambda_i|\leq 1$ for $1\leq i\leq N$, we are able to show that the operator is frequently hypercyclic. To achieve this we prove that $C_\phi\circ D^\alpha$ is Runge transitive.
%Namely, this cases are the ones that correspond to the case $b)$ of the former theorem. In some of this cases we can prove that the operator is frequently hypercyclic, trough the concept of Runge transitivity which was introduced by Bonilla and Grosse-Erdmann in \cite{BonGro07}.

\begin{definition}\rm\label{runge transitive}
An operator $T$ on a Fr\'echet space $X$ is called Runge
transitive if there is an increasing sequence $(p_n)$ of seminorms defining the topology of $X$
and numbers $N_m \in \N$, $C_{m,n} > 0$ for $m,n \in \N$ such that:
\begin{enumerate}
\item for all $m,n \in \N$ and $x \in X$,
$$
p_m (T^n x) \leq C_{m,n} p_{n+N_m} (x)
$$
\item for all $m,n \in \N$, $x,y \in X$ and $\varepsilon > 0$ there is some $z \in X$ such that
$$
p_n (z - x) < \varepsilon \textrm{ and } p_m (T^{n+N_m} z - y) < \varepsilon.
$$
\end{enumerate}
\end{definition}

The concept of Runge transitivity was introduced by Bonilla and Grosse-Erdmann, were they proved in \cite[Theorem 3.3]{BonGro07}, that every Runge transitive operator on a Fr\'echet space is frequently hypercyclic. They also show that every translation operator on $H(\C)$ is Runge transitive. However, the differentiation operator on $H(\C)$ is not Runge transitive, even though we know that it is strongly mixing in the Gaussian sense. Now, we prove that some of the operators which are included in the case $b)$ are frequently hypercyclic.

\begin{proposition}
Let $T$ be the operator on $H(\C^N)$, defined by $Tf(z)=C_\phi\circ D^\alpha f(z)$, with $\alpha\neq 0$, $\phi(z)=(\lambda_1z_1+b_1,\dots,\lambda_Nz_N+b_N)$ and $\lambda_i\neq 0$ for all $i$, $1\le i\le N$. Then, if $|\lambda_i|\leq 1$ for every $i$, $1\le i\le N$ and we have that $b_j\ne0$ and $\lambda_j=1$ for some $j$, $1\le j\le N$, then $T$ is Runge transitive.
\end{proposition}

\begin{proof}
Define the increasing sequence of seminorms
$$
p_m(f)=\sup_{\prod_{i=1}^N B(0,r_i(m))}|f(z)|,
$$
where the radius $r_i(m)$ are defined as follows:
$$
r_i(m)= \left\{
	\begin{array}{cc}
		|b_i|m  & \mbox{if } b_i \neq 0 \\
		 m & \mbox{if } b_i = 0
	\end{array}
\right.
$$
We will prove that both conditions of the Definition \ref{runge transitive} are satisfied with $N_m=m+1$.
For the first condition, we proceed as in the proof of part $c)$ of Theorem \ref{nonconvolution CN}. We will apply several times the Cauchy inequalities (\ref{cauchy}) with $\varepsilon_i$ defined as
$$
\varepsilon_i= \left\{
	\begin{array}{cc}
		\frac{|b_i|}{2}  & \mbox{if } b_i \neq 0 \\
		\frac{1}{2} & \mbox{if } b_i = 0
	\end{array}
\right.
$$
and in each step we divide it by 2. So, we get that
$$
p_m(T^nf) \leq \frac{2^{(n(n+1)/2)(|\alpha|+N)}\alpha!^n}{ {(2\pi)}^{nN}  \varepsilon_1^{n(\alpha_1+1)}\dots\varepsilon_N^{n(\alpha_N+1)}} \sup_{\Lambda} |f(z)|,
$$
where $\Lambda= \prod B\left(\phi_i^n(0),|\lambda_i|^nr_i(m)+\varepsilon_i\sum_{k=0}^{n-1}\frac{|\lambda_i|^k}{2^{n-k-1}}\right)$.

Since $|\lambda_i|\leq 1$ for every $i$, $1\le i\le N$, we obtain that
$$
|\phi_i^n(0)|=\left|b_i\sum_{k=0}^{n-1}\lambda_i^k\right| \leq |b_i|n,
$$
and that
$$
|\lambda_i|^nr_i(m)+\varepsilon_i\sum_{k=0}^{n-1}\frac{|\lambda_i|^k}{2^{n-k-1}} \leq r_i(m) + 2\varepsilon_i.
$$
From here it is easy to prove that $\Lambda \subseteq \prod B(0, r_i(n+m+1))$. Thus, if we denote $C_{m,n} = \frac{2^{(n(n+1)/2)(|\alpha|+N)}\alpha!^n}{ {(2\pi)}^{nN}  \varepsilon_1^{n(\alpha_1+1)}\dots\varepsilon_N^{n(\alpha_N+1)}}$, we get that
$$
p_m(T^nf) \leq C_{m,n} p_{n+m+1} (f).
$$

Suppose that $\varepsilon$ is a positive number, $n$ and $m$ are two integer numbers and that $f$, $g$ are two holomorphic functions on $H(\C^N)$, we want to prove that there exists some function $h\in H(\C^N)$ such that
$$
p_n (f - h) < \varepsilon \textrm{ and } p_m (T^{n+m+1}h - g) < \varepsilon.
$$
Similarly, for the second condition we can estimate $p_m(T^{n+m+1}h-g)$ in the same way we did previously by making use of the right inverse for $T$. We get that
$$%\begin{align*}
p_m(T^{n+m+1}h-g)\leq C \sup_{\Gamma} |S^{n+m+1}g - h|
$$%\end{align*}
where $C$ is some positive constant and $\Gamma = \prod B\left(\phi_i^{n+m}(0),|\lambda_i|^{n+m+1}r_i(m)+\varepsilon_i\sum_{k=0}^{n+m}\frac{|\lambda_i|^k}{2^{n-k-1}}\right)$.

To assure the existence of such function $h$, by Runge's Theorem \ref{Runge}, it is enough to prove that $\Gamma\cap \prod B(0,r_i(n))=\emptyset$.
We study this sets in the $j$-th coordinate. We get that

$$
\Gamma_j = B(b_j(n+m), r_j(m)+2\varepsilon_j) = B(b_j(n+m), |b_j|(m+1)),
$$
which is disjoint from $B(0, |b_j| n)$. Then, we have proved that the operator $T$ is Runge transitive, hence it is frequently hypercyclic.
\end{proof}

%%%%%%%%%%%%%%%%%%%%%%%%%%%%%%%%%%%%%%%%%%%%%%%%%%%%%%%%%%%%
%SECTION 4

\section{The non-diagonal case}

We are now interested in the case in which the automorphism $\phi(z)=Az + b,$ is given by any invertible matrix $A \in \C^{N\times N}$. Let $v\neq 0$ be any vector in $\C^N$ and let $T$ be the operator on $H(\C^N)$ defined by
$$
Tf(z)=C_\phi\circ D_v f (z)=D_vf(Az+b),
$$
where $D_vf$ is the differential operator in the direction of $v$,
$$D_vf(z_0)=\lim_{s\to 0} \frac{f(z_0+sv)-f(z_0)}{s}= \nabla f(z_0)\cdot v= df(\phi(z_0))(v).$$
%Note that if $v$ is a canonical vector of $\C^N$, then $D_v=D^v$.
The next two remarks show that we may consider a simplified version of the operator $T$.
\begin{remark}\rm\label{conjugar.la.matriz}
We can assume that the matrix $A$ is given in its Jordan form. Indeed, let $Q$ be an invertible matrix
in $\C^{N\times N}$ such that $A=QJQ^{-1}$, where $J$ is the Jordan form of $A$. Also let $c=Q^{-1}b$ and
denote $Q^*(f)(z)=f(Qz)$ for $f\in H(\C^N)$. Thus, we have that
$$
Q^*(C_\phi\circ D_v f)(z)=\nabla f(AQz+b)\cdot v.
$$
If we denote $\psi(z)=Jz+c$ and $w=Q^{-1}v$ then,
$$
(C_\psi\circ D_w) Q^*(f)(z)= \nabla f(Q(Jz+c))\cdot Qw = \nabla f(AQz+b)\cdot v.
$$
We have proved that the following diagram commutes

	$$
    \xymatrix{H(\C^N) \ar[r]^{C_\phi \circ D_v} \ar[d]_{Q^*} & H(\C^N) \ar[d]^{Q^*}  \\ H(\C^N) \ar[r]_{C_\psi\circ D_w} & H(\C^N)
}
    $$
This shows that $C_\phi \circ D_v$ is linearly conjugate to $C_\psi\circ D_w$.
\end{remark}

\begin{remark}\rm\label{conjugar.el.b}
We can assume that $b=0$ if the affine linear map $\phi$ has a fixed point $z_0=\phi(z_0)$.
Indeed, if we denote $\varphi(z)=Az$ then,
$$
\tau_{z_0}(C_\phi \circ D_v) (f)(z)= D_v(f)(A(z+z_0)+b)= \tau_{z_0}D_v(f)(Az) = (C_\varphi \circ
D_v)\tau_{z_0} (f) (z).
%\nabla f(A(z+z_0)+b)\cdot v = \nabla f(Az+z_0)\cdot v
$$
We have that the following diagram commutes
$$
    \xymatrix{H(\C^N) \ar[r]^{C_\phi \circ D_v} \ar[d]_{\tau_{z_0}} & H(\C^N) \ar[d]^{\tau_{z_0}} \\ H(\C^N) \ar[r]_{C_\varphi\circ D_v} & H(\C^N)
 }
$$
We conclude that $C_\phi \circ D_v$ is linearly conjugate to $C_\varphi\circ D_v$.
\end{remark}
The first two results of this section deal with affine transformations that have fixed points.

\begin{proposition}
Let $A \in \C^{N\times N}$ be an invertible matrix and let $v$ be a nonzero vector in $\C^N$. Suppose that the affine linear map $\phi(z)=Az+b$ has a fixed point and that
$$
\lim_{k\to\infty} k! \prod_{i=0}^{k-1} \|A^i v\| <+\infty.
$$
Then the operator $C_\phi\circ D_v$ acting on $H(\C^N)$ is not hypercyclic.

Consequently, $C_\phi\circ D_v$ is not hypercyclic if $v$ belongs to an invariant subspace $M$ of $A$ such that the spectral radius of the restriction, $r(A|_M)$, is less than 1. This happens in particular if $r(A)<1$ or if $v$ is an eigenvector of $A$ associated to an eigenvalue of modulus strictly less than 1.
\end{proposition}
\begin{proof}
We denote by $d^kf(z)$ to the $k$-th differential of a function $f$ at $z$, which is a $k$-homogenous polynomial, and we denote by $\big(d^kf\big)^\vee(z)$ to the associated symmetric $k$-linear form.

It is not difficult to see that the orbits of the operator $C_\phi\circ D_v$ are determined by
$$
(C_\phi\circ D_v)^kf (z)= \big(d^kf \big)^\vee(\phi^kz)(v,Av,\dots,A^{k-1}v).
$$
Assume that $z_0$ is a fixed point of $\phi$, then applying the Cauchy's inequalities we get
\begin{align*}
|(C_\phi\circ D_v)^kf (z_0)|&= |\big(d^kf\big)^\vee(\phi^kz_0)(v,Av,\dots,A^{k-1}v)| = |\big(d^kf\big)^\vee(z_0)(v,Av,\dots,A^{k-1}v)|\\
&\leq k! \prod_{i=0}^{k-1} \|A^i v\| \sup_{|z-z_0|<1}|f(z)|.
\end{align*}

Therefore $\{(C_\phi\circ D_v)^k f (z_0)\}$ is a bounded set of $\zC$. Since the evaluation at $z_0$ is continuous, $C_\phi\circ D_v$ cannot have dense orbits.

For the last assertion, first note that if $J=Q^{-1}AQ$ is the Jordan form of $A$, we have that $w=Q^{-1}v$ belongs to the invariant subspace $Q^{-1}M$ of $J$ and that $r:=r(J|_{Q^{-1}M})<1$. By Remarks \ref{conjugar.la.matriz} and \ref{conjugar.el.b} it suffices to prove that $C_J\circ D_w$ is not hypercyclic.

It is not difficult to show that for every $i\ge N$,
$$
\|J^iw\|\le cr^{i-N}i^N\|w\|,
$$
where $c$ is a constant that depends only on $r$ and $N$. Therefore,
\begin{align*}
 k! \prod_{i=0}^{k-1} \|J^i w\|  &\le  k! \prod_{i=0}^{N-1}\|J^i w\| \prod_{i=N}^{k-1}cr^{i-N}i^N\|w\| \\
&\le  (k!)^{N+1}\|J\|^{(N+1)N/2}c^{k-N}\|w\|^kr^{(k-N)(k-N-1)/2} \to 0,
\end{align*}
which implies that $C_J\circ D_w$ is not hypercyclic by the first part of the proposition.
\end{proof}

%\begin{remark}\rm
%There are two simple situations where the hypothesis of the previous proposition are satisfied, if the affine linear map $\phi(z)=Az+b$ has a fixed point. The first one is when $\|A\|<1$,
%$$
%k! \prod_{i=0}^{k-1} \|A^i v\| \leq k! \|v\|^k \prod_{i=0}^{k-1} \|A\|^i = k! \|v\|^k \|A\|^{\frac{k(k-1)}{2}} \to_{k\to\infty} 0.
%$$
%The second one is when $v$ is an eigenvector of $A$ associated to an eigenvalue $\lambda$ of modulus strictly less than 1,
%$$
%k! \prod_{i=0}^{k-1} \|A^i v\| = k! \prod_{i=0}^{k-1} |\lambda^i| \|v\| = k! \|v\|^k |\lambda|^{\frac{k(k-1)}{2}} \to_{k\to\infty} 0.
% $$
%\end{remark}
In opposition to the previous result, if the matrix $A$ is expansive when restricted to an invariant subspace then the operator is strongly mixing in the Gaussian sense. This assumption is similar to the hypothesis of the results in the previous sections.
%One may think that this assumption is too much restrictive, but actually it is very close to the hypothesis we studied in the previous sections. In fact,
Indeed, in the one dimensional case we have that $\phi(z)=\lambda z + b$ and if $|\lambda|\geq 1$, then the operator $C_\phi \circ D$ is strongly mixing in the Gaussian sense. Here, the linear part of the composition operator is expansive. This situation still holds in the diagonal case in $H(\C^N)$. In this last case, we have that $\phi(z_1,\dots,z_N)=(\lambda_1 z_1 + b_1,\dots,\lambda_N z_N + b_N)$. Suppose that $\alpha$ is a multi-index of modulus one, i.e. that $D^\alpha$ is a partial derivative, then the hypothesis $|\lambda^\alpha|\geq 1$ turns out to be exactly the same as imposing that the linear part of $\phi$ is expansive on the subspace spanned by $\alpha$. The proper result reads as follows.

\begin{proposition}
Let $A \in \C^{N\times N}$ be an invertible matrix and let $v\neq 0$ be a vector in $\C^N$. Suppose that the affine linear map $\phi(z)=Az+b$ has a fixed point and that $v$ belongs to a subspace $M$ that reduces $A$ and such that $\|(A|_M)^{-1}\|<1$. Then the operator $C_\phi\circ D_v$ acting on $H(\C^N)$ is strongly mixing in the Gaussian sense.
\end{proposition}
\begin{proof}

%Let us denote $J$ the Jordan form of $A$, $Q$ be the change of variables that takes $A$ to it's Jordan form and $w=Q^{-1}v$. By Remarks \ref{conjugar.la.matriz} and \ref{conjugar.el.b}, it is enough to prove that $T=C_J\circ D_w$ is hypercyclic. Note that $w$ belongs to the invariant subspace $Q^{-1}M$ of $J$, and every eigenvalue of $J|_{Q^{-1}M}$ has modulus bigger than 1.
%
We will show that the hypothesis of the Theorem \ref{murillo} are fulfilled, taking as dense sets the polynomials in $N$ complex variables. It is clear that $\sum_n T^nf$ converges unconditionally for every polynomial $f$. Now we will define a right inverse for $C_\phi\circ D_v$, but first we set some notation. Let us denote the fixed point of $\phi$ by $z_0$.
%Let $\{v,v_2,\dots,v_N\}$ be an orthogonal basis of $\C^N$ such that $\{v_1,\dots,v_K\}$ is a basis of $M$. %with $v_1=v$, and $\|v_j\|=\|v\|$ for every $j=1,\dots,N$.
Let us denote by $\pi_1$ to the orthogonal projection over $M$,  $\pi_2=I-\pi_1$ the orthogonal projection over $M^\perp$. Set $\mu(z)= \frac{\langle z,v\rangle}{\|v\|^2}$. We have that $z\mapsto\mu(z)v$ is the orthogonal projection over $span\{v\}$, and we denote $\tilde\pi= \pi_1 - \mu(z)v$. Finally, set $\phi_i(z)=Az + \pi_i(b)$, for $i=1,\, 2$. Since, $M$ reduces $A$, we have that $\phi_i$ is invertible and that $\pi_i(z_0)$ is a fixed point of $\phi_i$, for $i=1,\, 2$.

We define now for each $g\in H(\C^N)$,
$$
Rg(z)=\int_{\mu(z_0)}^{\mu(z)} g(\phi_1^{-1}(tv+\tilde\pi (z))+\pi_2 (z)) dt,
$$
and $C(g)(z)=g(\pi_1(z)+\phi_2^{-1}(\pi_2(z)))$. Note that $R\circ C=C\circ R$. Finally, let $S=C\circ R$. Observe that,
$$
Sg(z)=\int_{\mu(z_0)}^{\mu(z)} g(\phi^{-1}(tv+\tilde\pi (z)+\pi_2 (z))) dt.
$$
We have that
\begin{align*}
D_v Sg(z) &= \lim_{s\to 0} \frac{Sg(z+sv)-Sg(z)}{s}\\
&= \lim_{s\to 0} \frac{1}{s}\left[\int_{\mu(z_0)}^{\mu(z+sv)} g(\phi^{-1}(tv+\tilde\pi (z)+\pi_2(z))) dt - \int_{\mu(z_0)}^{\mu(z)} g(\phi^{-1}(tv+\tilde\pi (z)+\pi_2(z)))  dt \right]\\
&= \lim_{s\to 0} \frac{1}{s} \int_{\mu(z)}^{\mu(z)+s} g(\phi^{-1}(tv+\tilde\pi (z)+\pi_2(z)))  dt\\
&= g(\phi^{-1}(\mu(z)v +\tilde\pi (z)+\pi_2(z))) \\
&= g(\phi^{-1}z).
\end{align*}
Thus, $[C_\phi\circ D_v]\circ S g = g$ for every $g\in H(\C^N)$.
To conclude the proof we need to show that $\sum_n S^ng$ converges unconditionally for every polynomial $g$.

First we will bound the supremum of $|Rg|$ on $B(\pi_1 z_0,r)\times B(\pi_2 z_0, s)$, for a fixed polynomial $g$. Suppose that $z\in B(\pi_1 z_0,r)\times B(\pi_2 z_0, s)$ and that $t\in [\mu(z_0),\mu(z)]$ i.e. $t$ lives in the complex segment from $\mu(z_0)$ to $\mu(z)$. Then we have that
\begin{align*}
\|tv+\tilde\pi(z)-\pi_1 z_0\|^2 &= \|(t-\mu(z_0))v+\tilde\pi (z-z_0)\|^2 \\ &= |t- \mu(z_0)|^2 \|v\|^2 + \|\tilde\pi (z-z_0)\|^2 \\ &\leq |\mu(z)- \mu(z_0)|^2 \|v\|^2 + \|\tilde\pi (z-z_0)\|^2 + \\ &= \|\pi_1(z-z_0)\|^2 < r^2.
\end{align*}
Also, suppose that $\sigma:=\|(A|_M)^{-1}\|<1$. We get that
\begin{align*}
\|\phi_1^{-1}(\pi_1 (z))- \pi_1 (z_0)\| &=\|\phi_1^{-1}(\pi_1 (z))- \phi_1^{-1}(\pi_1 (z_0))\| \\ &= \|A^{-1}(\pi_1 (z)-\pi_1 (b))- A^{-1}(\pi_1 (z_0)- \pi_1 (b))\| \\ &\leq \left\|(A|_M)^{-1}\right\| \|\pi_1(z)-\pi_1 (z_0)\| = \sigma r.
\end{align*}

Gathering the previous statements we get that
\begin{align*}
|Rg(z)| &\leq |\mu(z)-\mu(z_0)| \sup_{t\in [\mu(z_0),\mu(z)]} |g(\phi_1^{-1}(tv+\tilde\pi (z))+\pi_2 (z))| \\
&\leq \frac{r}{\|v\|^2} \sup_{w\in B(\pi_1 z_0,r)\times B(\pi_2 z_0, s)} |g(\phi_1^{-1}(\pi_1 (w)) + \pi_2(w))| \leq
\frac{r}{\|v\|^2} \sup_{w\in B(\pi_1 z_0,\sigma r)\times B(\pi_2 z_0, s)} |g(w)|.
\end{align*}

Thus, we have proved that
$$
\sup_{B(\pi_1 z_0,r)\times B(\pi_2 z_0, s)}|Rg| \leq \frac{r}{\|v\|^2} \sup_{B(\pi_1 z_0,\sigma r)\times B(\pi_2 z_0, s)} |g|.
$$

Following by induction we obtain that
\begin{align*}
\sup_{B(\pi_1 z_0,r)\times B(\pi_2 z_0, s)}|R^ng| &\leq \frac{r}{\|v\|^2} \sup_{B(\pi_1 z_0,\sigma r)\times B(\pi_2 z_0, s)} |R^{n-1}g| \\ &\leq \frac{r^n}{\|v\|^{2n}} \sigma^{\frac{n(n-1)}{2}} \sup_{B(\pi_1 z_0,\sigma^n r)\times B(\pi_2 z_0, s)} |g| .
\end{align*}

Finally, to conclude the proof we compute $\sup_{B(\pi_1 z_0,r)\times B(\pi_2 z_0, s)}|S^n g(z)|$:

\begin{align*}
 \sup_{B(\pi_1 z_0,r)\times B(\pi_2 z_0, s)}|S^n g(z)| & = \sup_{B(\pi_1 z_0,r)\times B(\pi_2 z_0, s)}|R^n C^ng(z)| \\
 & \leq \frac{r^n}{\|v\|^{2n}} \sigma^{\frac{n(n-1)}{2}} \sup_{B(\pi_1 z_0,\sigma^n r)\times B(\pi_2 z_0, s)}|C^n g(z)| \\
 & \leq \frac{r^n}{\|v\|^{2n}} \sigma^{\frac{n(n-1)}{2}} \sup_{B(\pi_1 z_0,\sigma^n r)\times \phi_2^{-n} (B(\pi_2 z_0, s))}|g(z)|
\end{align*}

Since $\sigma<1$, we have proved that $\sum_n S^ng$ converges unconditionally for every polynomial $g$. Hence the operator $C_\phi\circ D_v$ is strongly mixing in the Gaussian sense, as we wanted to prove.
\end{proof}

We turn now our discussion to the cases in which the affine linear map $\phi(z)=Az+b$ does not have a fixed point. This is equivalent to say that $b\notin Ran(I-A)$.
Thus, $1$ belongs to the spectrum of $A$. Then the Jordan form of $A$, which we denote by $J$, has a sub-block with ones in the principal diagonal and the first sub-diagonal and zeros elsewhere. It is easy to see that there exists
some $k\in\zN$, $k\leq N$ such that the canonical vector $e_k$ does not belong to $Ran(I-J)$ and such that
$b_k\not = 0$. This argument will be the key to show that $\phi$ is a runaway map, hence the operator $C_\phi
\circ D_v$ is topologically transitive. The proof of this result is in the spirit of part $b)$ of Theorem
\ref{nonconvolution CN}.

\begin{proposition}
Let $A \in \C^{N\times N}$ be an invertible matrix and let $v\neq 0$ be a vector in $\C^N$. Suppose that the affine linear map $\phi(z)=Az+b$ does not have a fixed point. Then the operator $C_\phi\circ D_v$ acting on $H(\C^N)$ is mixing.
\end{proposition}
\begin{proof}
Due to the previous observations it is enough to prove that $C_\psi\circ D_w$ is topologically transitive if $\psi(z)=Jz+b$ with $b\notin Ran (I-J)$ and $w\in\zC^N$, $w\not =0$. We will denote $T=C_\psi\circ D_w$.

Given $K_U$, $K_V$ two compact sets of $\zC^N$, $h_U$, $h_V$ two holomorphic functions in $H(\zC^N)$ and $\theta$ a positive real number, we want to prove that there exists $k\in\zN$ and $g\in H(\zC^N)$ such that

\begin{equation}\label{T cumple Runge}
\|g-h_U\|_{K_U}<\theta \; \textrm{  and  } \; \|(C_\psi\circ D_w)^kg-h_V\|_{K_V}<\theta.
\end{equation}

We will use Runge's theorem to show the existence of such function $g$. As before, we denote by $S$ the right inverse of $D_w$. We have that
\begin{align*}
\sup_{K_V} \left|C_\psi\circ D_w g(z)-h_V(z)\right| &= \sup_{K_V} \left|C_\psi \left(D_w g(z)- C_{\psi^{-1}} h_V(z)\right)\right| \\
&= \sup_{C_\psi (K_V)} \left|D_w g(z)- C_{\psi^{-1}} h_V(z))\right| \\
& =  \sup_{J(K_V) + b } \left|D_w \left(g(z)- S\circ C_{\psi^{-1}} h_V(z)\right)\right| \\
& \leq \frac{\|w\| N}{\varepsilon_1^N} \sup_{J(K_V) + B_{\varepsilon_1}(b) } \left|g(z)- S\circ C_{\psi^{-1}} h_V(z)\right| .
\end{align*}

Following in this way inductively, we will get an estimate of $\|(C_\psi\circ D_w)^kg-h_V\|_{K_V}$,

$$
\sup_{K_V} \left|(C_\psi\circ D_w)^l g(z)-h_V(z)\right| \leq \alpha(l) \sup_{A_l} \left|g(z)- (S\circ C_{\psi^{-1}})^l h_V(z)\right|,
$$

with $\alpha(l)>0$ and $A_l= J^l(K_V)+ \sum_{i=1}^l J^i(B(0,\varepsilon_i)) + \sum_{i=1}^l J^i(b)$.

It is enough to find some $l\in\zN$ such that $K_U\cap A_l= \emptyset$. Without loss of generality we can assume that $e_1\notin Ran(J-I)$ and $b_1\not = 0$ (see the comments before the proposition). This means that $J$ acts like the identity in the first coordinate.

Suppose that $K_V\subset \prod_{i=1}^{N}B(0,r_i)$, then if we project in the first coordinate and choose proper $\varepsilon_i>0$ we obtain

\begin{align*}
[A_l]_1 &= [J^l(K_V)]_1+ \sum_{i=1}^l [J^i(B(0,\varepsilon_i))]_1 + \sum_{i=1}^l [J^i(b)]_1 \\
&\subset B(0,r_1) + B(0, \sum_{i=1}^l \varepsilon_i) + lb_1 \\
&\subset B(0, R) + lb_1.
\end{align*}

Thus, we will able to find $l_0\in\zN$ such that $[K_U]_1\cap [A_l]_1= \emptyset$ for all $l\geq l_0$. Therefore, by Runge's Theorem, there exists some $g_l\in H(\zC^N)$ such that (\ref{T cumple Runge}) is satisfied for all $l\geq l_0$. We have proved that the operator $C_\psi\circ D_w$ is mixing, as we wanted to prove.

\end{proof}

\bibliography{biblio}
\bibliographystyle{plain}

\end{document}